\documentclass[leqno,12pt]{article}
\usepackage{amsmath, amssymb}
\usepackage{amsthm}

\theoremstyle{plain} 
\newtheorem{theorem}{\indent\sc Theorem}[section]
\newtheorem{lemma}[theorem]{\indent\sc Lemma}
\newtheorem{corollary}[theorem]{\indent\sc Corollary}

\theoremstyle{definition} 
\newtheorem{definition}[theorem]{\indent\sc Definition}

\newtheorem{example}[theorem]{\indent\sc Example}


\begin{document}

\title{\uppercase{Generalized Kenmotsu Manifolds}}
\author{ {\small {}
\bigskip} \\
\textsc{Aysel TURGUT\ VANLI and Ramazan SARI} }
\date{}
\maketitle

\begin{abstract}
In 1972, K. Kenmotsu studied a class of almost contact Riemannian manifolds.
Later, such a manifold was called a Kenmotsu manifold. This paper, we
studied Kenmotsu manifolds with $(2n+s)$-dimensional $s-$contact metric
manifold and this manifold, we have called generalized Kenmotsu manifolds.
Necessary and sufficient condition is given for an almost $s-$contact metric
manifold to be a generalized Kenmotsu manifold.We show that a generalized Kenmotsu
manifold is a locally warped product space. In addition, we study some curvature properties of
generalized Kenmotsu manifolds. Moreover, we show that the $\varphi $%
-sectional curvature of any semi-symmetric and projective semi-symmetric $%
(2n+s)$-dimensional generalized Kenmotsu manifold is $-s$.
\end{abstract}

\footnote{
} 
\section*{Introduction}

In \textit{\cite{Y},} K.Yano introduced the notion of a $f-$structure on a
differentiable manifold $M$, i.e., a tensor fields $f$ of type $(1,1)$ and $%
rank$ $2n$ satisfying $f^{3}+f=0$ as a generalization of both (almost)
contact (for $s=1$) and (almost) complex structures (for $s=0$). $TM$ splits
into two complementary subbundles $\mathcal{L=}$ $Im\varphi $ and $\mathcal{%
M=}$ $ker\varphi $. The existence of which is equivalent to a reduction of
the structural group of the tangent bundle to $\mathit{U(n)\times O(s)}$
\textit{\cite{BL}}. H. Nakagawa in \textit{\cite{Nak}} and \textit{\cite{Nak2}} introduced the notion of globally framed f-manifolds (\textit{called f-manifolds}), later developed and studied by Goldberg and Yano \textit{\cite{G},} \textit{\cite{GY},} \textit{\cite{GY2}}. A wide class of globally frame $f$-manifolds was introduced in \textit{\cite{BL},} by Blair according to the following definition. A
metric $f$-structure is said to be a $K$-structure if the fundamental 2-form
$\Phi $, defined usually as $\Phi (X,Y)=g(X,\varphi Y)$, for any vector
fields $X$ and $Y$ on $M$, is closed and the normality condition holds, that
is; $[\varphi ,\varphi ]+2\sum_{i=1}^{s}d\eta ^{i}\otimes \xi _{i}=0$ where $%
[\varphi ,\varphi ]$ denotes the Nijenhuis torsion of $\varphi $. Some authors studeid $f$-structure \textit{\cite{BL2},} \textit{\cite{CFF},} \textit{\cite{YK}}. The
Riemannian connection $\nabla $ of a metric $f-$manifold satisfies the
following formula \textit{\cite{Bkitap},}
\begin{eqnarray}
2g((\nabla _{X}\varphi )Y,Z) &=&3d\Phi (X,\varphi Y,\varphi Z) -3d\Phi(X,Y,Z)+g(N^{1}(Y,Z),\varphi X)  \notag \\
&&+\overset{s}{\underset{i=1}{\sum }}\{N^{2}(Y,Z)\eta ^{i}(X)+2d\eta
^{i}(\varphi Y,X)\eta ^{i}(Z)-2d\eta ^{i}(\varphi Z,X)\eta ^{i}(Y)\},
\end{eqnarray}
where the tensor fields $N^{1}$ and $N^{2}$ are defined by $N^{1}=[\varphi
,\varphi ]+2\sum_{i=1}^{s}d\eta ^{i}\otimes \xi _{i},$ $N^{2}(X,Y)=(L_{_{%
\varphi X}}\eta ^{i})(Y)-(L_{_{\varphi Y}}\eta ^{i})(X)$ respectively, which
is by a simple computation can be rewritten as: $N^{2}(X,Y)=2d\eta
^{i}(\varphi X,Y)-2d\eta ^{i}(\varphi Y,X).$

Let $M$ be a $(2n+1)$ dimensional differentiable manifold. $M$ is called an
\textit{almost contact metric manifold} if $\varphi $ is $(1,1)$ type tensor
field, $\xi $ is vector field, $\eta $ is $1-$ form and $g$ is a compatible
Riemannian metric such that
\begin{equation}
\varphi ^{2}=-I+\eta \otimes \xi ,%
\begin{array}{cc}
&
\end{array}%
\eta (\xi )=1
\end{equation}%
\begin{equation}
g(\varphi X,\varphi Y)=g(X,Y)-\eta (X)\eta (Y)
\end{equation}%
for all $X,Y\in \Gamma (TM).$

In addition, we have%
\begin{equation}
\eta (X)=g(X,\xi ),%
\begin{array}{cc}
&
\end{array}%
\varphi (\xi )=0,%
\begin{array}{cc}
&
\end{array}%
\eta (\varphi )=0
\end{equation}%
for all $X,Y\in \Gamma (TM)$ \cite{Bkitap}.

\bigskip To study manifolds with negative curvature, Bishop and O'Neill
introduced the notion of warped product as a generalization of Riemannian
product \textit{\cite{Bis}}. In 1960's and 1970's, when almost contact
manifolds were studied as an odd dimensional counterpart of almost complex
manifolds, the warped product was used to make examples of almost contact
manifolds \textit{\cite{Tan}}. In addition, S. Tanno classified the
connected $(2n+1)$ dimensional almost contact manifold $M$ whose
automorphism group has maximum dimension $(n+1)^{2}$ in \textit{\cite{Tan}}$%
. $ For such a manifold, the sectional curvature of plane sections
containing $\xi $ is a constant, say $c$. Then there are three classes.

$i)$ $c>0$, M is homogeneous Sasakian manifold of constant holomorphic
sectional curvature.

$ii)$ $c=0$, M is the global Riemannian product of a line or a circle with a
K$\ddot{a}$hler manifold of constant holomorphic sectional curvature.

$iii)$ $c<0$, M is warped product space $%
\mathbb{R}
\times _{f}%
\mathbb{C}
^{n}.$ Kenmotsu obtained some tensorial equations to characterize manifolds
of the third case.

In 1972 , Kenmotsu abstracted the differential geometric properties of the
third case. In \textit{\cite{K},} Kenmotsu studied a class of almost contact
Riemannian manifold which satisfy the following two condition,%
\begin{eqnarray}
(\nabla _{X}\varphi )Y &=&-\eta (Y)\varphi X-g(X,\varphi Y)\xi \\
\nabla _{X}\xi &=&X-\eta (X)\xi  \notag
\end{eqnarray}

He showed normal an almost contact Riemannian manifold with $(5)$ but not
quasi Sasakian hence not Sasakian. He was to characterize warped product
space $L\times _{f}%
\mathbb{C}
E^{n}$ by an almost contact Riemannian manifold with $(5)$. Moreover, he
show that every point of an almost contact Riemannian manifold with $(5)$
has a neighborhood which is a warped $(-\epsilon ,\epsilon )\times _{f}V$
where $f(t)=ce^{t}$ and $V$ is K\"{a}hler.

In 1981 \textit{\cite{Jan},} Janssens and Vanhecke, an almost contact metric
manifold satisfiying this $(5)$ is called a Kenmotsu manifold. After this
definition, some authors studied Kenmotsu manifold \textit{\cite{DP}, \cite{JDP}, \cite{P}, \cite{Pr}.}

\bigskip The paper is organized as follows: after a preliminary basic
notions of $s-$contact metric manifolds theory, in Section 2, we introduced
generalized almost Kenmotsu manifolds and generalized Kenmotsu manifolds.
Necessary and sufficient condition is given for a $s-$contact metric
manifold to be a generalized Kenmotsu manifold. The warped product $%
L^{s}\times _{f}V^{2n}$ provides an example. In section 3, some curvature
properties are given. In section 4, we studied Ricci curvature tensor. In section 5,
we studied semi-symmetric properties of generalized Kenmotsu manifolds. We
show that the $\varphi $-sectional curvature of any semi-symmetric and
projective semi-symmetric $(2n+s)$-dimensional generalized Kenmotsu manifold
is $-s$.

\section{Preliminaries}

In \textit{\cite{GY},} a $(2n+s)-$dimensional differentiable manifold $M$ is
called metric $f-$manifold if there exist an $(1,1)$ type tensor field $%
\varphi $, $s$ vector fields $\xi _{1},\dots ,\xi _{s}$, $s$ $1$-forms $\eta
^{1},\dots ,\eta ^{s}$ and a Riemannian metric $g$ on $M$ such that%
\begin{equation}
\varphi ^{2}=-I+\overset{s}{\underset{i=1}{\sum }}\eta ^{i}\otimes \xi _{i},%
\begin{array}{cc}
\begin{array}{c}
\end{array}
&
\end{array}%
\eta ^{i}(\xi _{j})=\delta _{ij}
\end{equation}%
\begin{equation}
g(\varphi X,\varphi Y)=g(X,Y)-\overset{s}{\underset{i=1}{\sum }}\eta
^{i}(X)\eta ^{i}(Y),
\end{equation}%
for any $X,Y\in \Gamma (TM),$ $i,j\in \{1,\dots ,s\}$. In addition, we have%
\begin{equation}
\eta ^{i}(X)=g(X,\xi _{i}),%
\begin{array}{cc}
\begin{array}{c}
\end{array}
&
\end{array}%
g(X,\varphi Y)=-g(\varphi X,Y),%
\begin{array}{cc}
&
\end{array}%
\varphi \xi _{i} =0,%
\begin{array}{cc}
&
\end{array}%
\eta ^{i}\circ \varphi =0.
\end{equation}

Then, a $2$-form $\Phi $ is defined by $\Phi (X,Y)=g(X,\varphi Y)$, for any $%
X,Y\in \Gamma (TM)$, called the \textit{fundamental }$\mathit{2}$\textit{%
-form}.

In what follows, we denote by $\mathcal{M}$ the distribution spanned by the
structure vector fields $\xi _{1},\dots ,\xi _{s}$ and by $\mathcal{L}$ its
orthogonal complementary distribution. Then, $TM=\mathcal{L}\oplus \mathcal{M%
}$. If $X\in \mathcal{M}$ we have $\varphi X=0$ and if $X\in \mathcal{L}$ we
have $\eta ^{i}(X)=0$, for any $i\in \{1,\dots ,s\}$; that is, $\varphi
^{2}X=-X$.

In a metric $f$-manifold, special local orthonormal basis of vector fields
can be considered. Let $U$ be a coordinate neighborhood and $E_{1}$ a unit
vector field on $U$ orthogonal to the structure vector fields. Then, from $%
(6)-(8)$, $\varphi E_{1}$ is also a unit vector field on $U$ orthogonal
to $E_{1}$ and the structure vector fields. Next, if it is possible, let $%
E_{2}$ be a unit vector field on $U$ orthogonal to $E_{1}$, $\varphi E_{1}$
and the structure vector fields and so on. The local orthonormal basis
\begin{equation*}
\{E_{1},\dots ,E_{n},\varphi E_{1},\dots ,\varphi E_{n},\xi _{1},\dots ,\xi
_{s}\},
\end{equation*}%
so obtained is called an $f$\textit{-basis}. Moreover, a metric $f$-manifold
is \textit{normal} if
\begin{equation*}
\lbrack \varphi ,\varphi ]+2\underset{i=1}{\overset{s}{\sum }}d\eta
^{i}\otimes \xi _{i}=0,
\end{equation*}%
where $[\varphi ,\varphi ]$ is denoting the Nijenhuis tensor field
associated to $\varphi $.

In \textit{\cite{Van},} \ let $M$ a $(2n+s)-$dimensional metric $f-$%
manifold. If there exists $2$-form $\Phi $ such that $\eta ^{1}\wedge
...\wedge \eta ^{s}\wedge \Phi ^{n}\neq 0$ on $M$ then $M$ is called an
\textit{almost s-contact metric manifold.} A normal almost s-contact metric manifold is called an s-contact metric manifold.
\section{Generalized Kenmotsu Manifolds}
As is known in Kenmotsu manifold $dimker\varphi =1$, since $ker\varphi
=sp\{\xi \}$. It was to be $dimker\varphi >1$\ open question.\ Firstly in
2003, L. Bhatt and K. K. Dube introduced\textit{\ Kenmotsu }$s$\textit{%
-structure}; that is, an almost $s-$contact metric manifold $M$ is called a
Kenmotsu $s-$manifold if \textit{\ }
\begin{equation*}
(\nabla _{X}\varphi )Y=g(\varphi X,Y)\overset{s}{\underset{i=1}{\sum }}\xi
_{i}-\varphi X\overset{s}{\underset{i=1}{\sum }}\eta ^{i}(Y)
\end{equation*}%
for any $X,Y\in \Gamma (TM)$ \textit{\cite{Bd}}$.$ We will give their definition as a theorem in this paper.

Afterwards in 2006, M. Falcitelli and A.M. Pastore introduced \textit{
Kenmotsu f.pk-manifold}. In \textit{\cite{FP}}, \textit{\ }a metric $f.pk$%
-manifold $M$ of dimension $2n+s$, $s\geq 1$, with $f.pk-$structure which is
a metrik $f-$structure with parallelizable kernel. $\left( \varphi ,\xi
_{i},\eta ^{i},g\right) $ is said to be a Kenmotsu $f.pk$-manifold if it is
normal, the$1$-forms $\eta ^{i}$ are closed and $d\Phi =2\eta ^{1}\wedge
\Phi $ \textit{.} They assume that $d\Phi =2\eta ^{i}\wedge \Phi $ for all $%
i=1,2,...,s$ in the definition of Kenmotsu $f.pk-$manifold. So, they remark
that, since the $1$-forms $\eta ^{i}$ are linearly independent and $\eta
^{i} $ $\wedge \Phi =\eta ^{j}\wedge \Phi $ implies $\eta ^{i}=\eta ^{j}$,
then the condition on $d\Phi $ can be satisfied by a unique $\eta ^{i}$ and
they can assume that $d\Phi =2\eta ^{1}\wedge \Phi .$ It is clear that
authors were equated 1-forms $\eta ^{1},...,\eta ^{s}$, which dual of $\xi
_{1},...,\xi _{s}.$Thus, they studied unique 1-form $\eta ^{1}.$

In this paper, all $\eta ^{1},...,\eta ^{s}$ $1-$forms are unequaled at the
definition of generalized Kenmotsu manifolds.
\begin{definition}
\label{Def} Let $M$ be an almost $s-$contact metric manifold of dimension $(2n+s)$, $s\geq 1$, with $%
\left( \varphi ,\xi _{i},\eta ^{i},g\right) $ . $M$ is said to be a
generalized almost Kenmotsu manifold if for all $1\leq i\leq s,$ $1-$forms $%
\eta ^{i}$ are closed and $d\Phi =2\underset{i=1}{\overset{s}{\sum }}\eta
^{i}\wedge \Phi .$
A normal generalized almost Kenmotsu manifold $M$ is called a generalized
Kenmotsu manifold.
\end{definition}
Now, we construct an example of generalized Kenmotsu manifold.

\begin{example}
We consider $(2n+s)-$dimensional manifold
\begin{equation*}
M=\left\{ (x_{1},...,x_{n},y_{1},...,y_{n},z_{1,}...,z_{s})\in \mathbb{R}%
^{2n+s}:\sum_{\alpha =1}^{s}{z_{\alpha }}\neq 0\right\}
\end{equation*}%
We choose the vector fields%
\begin{equation*}
X_{i}=e^{-\sum\limits_{\alpha =1}^{s}z_{\alpha }}\frac{\partial }{\partial
x_{i}},%
Y_{i}=e^{-\sum\limits_{\alpha =1}^{s}z_{\alpha }}\frac{\partial }{\partial
y_{i}},%
\xi _{\alpha }=\frac{\partial }{\partial z_{\alpha }},%
1\leq i\leq n,%
1\leq \alpha \leq s
\end{equation*}%
which are linearly indepent at each point of $M.$ Let $g$ be the Riemannian
metric defined by%
\begin{equation*}
g=e^{2\sum\limits_{\alpha =1}^{s}z_{\alpha }}\left[ \sum\limits_{i=1}^{n}%
\left( dx_{i}\otimes dx_{i}+dy_{i}\otimes dy_{i}\right) \right]
+\sum\limits_{\alpha =1}^{s}\eta ^{\alpha }\otimes \eta ^{\alpha }.
\end{equation*}%
Hence, $\left\{ X_{1},...,X_{n},Y_{1},...,Y_{n},\xi _{1},...,\xi
_{s}\right\} $ is an orthonormal basis. Thus, $\eta ^{\alpha }$ be the $1-$%
form defined by $\eta ^{\alpha }\left( X\right) =g(X,\xi _{\alpha }),%
\begin{array}{c}
\end{array}%
\alpha =1,...,s$ for any vector field $X$ on $TM.$ We defined the $(1,1)$
tensor field $\varphi $ as
\begin{equation*}
\varphi \left( X_{i}\right) =Y_{i},%
\varphi \left( Y_{i}\right) =-X_{i},%
\varphi \left( \xi _{\alpha }\right) =0,%
1\leq i\leq n,%
1\leq \alpha \leq s.
\end{equation*}%
The linearity property of $\varphi $ and $g$ yields that

\begin{eqnarray*}
\eta ^{\alpha }\left( \xi _{\beta }\right) &=&\delta _{\alpha \beta },%
\varphi ^{2}X=-X+\sum\limits_{\alpha =1}^{s}\eta ^{\alpha }\left( X\right)
\xi _{\alpha }, \\
g\left( \varphi X,\varphi Y\right) &=&g\left( X,Y\right)
-\sum\limits_{\alpha =1}^{s}\eta ^{\alpha }\left( X\right) \eta ^{\alpha
}\left( Y\right) ,
\end{eqnarray*}%
for any vector fields $X$, $Y$ on $M.$ Therefore,$(M,\varphi ,\xi _{\alpha
},\eta ^{\alpha },g)$ defines a metric $f-$manifold. We have $\Phi
(X_{i,},Y_{i})=-1$ and others are zero. Therefore, the essential non-zero
component of $\Phi $ is%
\begin{equation*}
\Phi (\frac{\partial }{\partial x_{i}},\frac{\partial }{\partial y_{i}})=g(%
\frac{\partial }{\partial x_{i}},\varphi \frac{\partial }{\partial y_{i}}%
)=-e^{2\sum\limits_{\alpha =1}^{s}z_{\alpha }}
\end{equation*}%
and hence, we have
\begin{equation*}
\Phi =-e^{2\sum\limits_{\alpha =1}^{s}z_{\alpha }}\sum_{i=1}^{n}dx_{i}\wedge
dy_{i}
\end{equation*}%
Therefore, we get $\eta ^{1}\wedge ...\wedge \eta ^{s}\wedge \Phi ^{n}\neq 0$
on $M$. Thus $(M,\varphi ,\xi _{\alpha },\eta ^{\alpha },g)$ is almost $s-$%
contact manifold. Consequently, the exterior derivative $d\Phi $ is given by
\begin{equation*}
d\Phi =2\sum\limits_{\alpha =1}^{s}dz_{\alpha }\wedge
(-e^{2\sum\limits_{\alpha =1}^{s}z_{\alpha }})\sum_{i=1}^{n}dx_{i}\wedge
dy_{i}.
\end{equation*}%
Therefore, $(M,\varphi ,\xi _{\alpha },\eta ^{\alpha },g)$ is a generalized
almost Kenmotsu manifold. It can be seen that $\ (M,\varphi ,\xi _{\alpha
},\eta ^{\alpha },g)$ is normal. So, it is a\ generalized Kenmotsu manifold.
Moreover, we get
\begin{eqnarray*}
\left[ X_{i},\xi _{\alpha }\right] &=&X_{i},%
\begin{array}{cc}
&
\end{array}%
\left[ Y_{i},\xi _{\alpha }\right] =Y_{i}, \\
\left[ X_{i},X_{j}\right] &=&0,%
\begin{array}{cc}
&
\end{array}%
\left[ X_{i},Y_{i}\right] =0,%
\begin{array}{cc}
&
\end{array}%
\left[ X_{i},Y_{j}\right] =0 \\
\left[ Y_{i},Y_{j}\right] &=&0,%
\begin{array}{cc}
&
\end{array}%
1\leq i,j\leq n,1\leq {\alpha }\leq s.
\end{eqnarray*}%
The Riemannian connection $\nabla $ of the metric $g$ is given by
\begin{eqnarray*}
2g(\nabla _{X}Y,Z) &=&Xg(Y,Z)+Yg(Z,X)-Zg(X,Y) \\
&&+g(\left[ X,Y\right] ,Z)-g(\left[ Y,Z\right] ,X)+g(\left[ Z,X\right] ,Y).
\end{eqnarray*}%
Using the Koszul's formula, we obtain
\begin{eqnarray*}
\nabla _{X_{i}}X_{i} &=&\sum_{\alpha =1}^{s}{\xi _{\alpha }},%
\begin{array}{c}
\end{array}%
\nabla _{Y_{i}}Y_{i}=\sum_{\alpha =1}^{s}{\xi _{\alpha }}, \\
\nabla _{X_{i}}X_{j} &=&\nabla _{Y_{i}}Y_{j}=\nabla _{X_{i}}Y_{i}=\nabla
_{X_{i}}Y_{j}=0 \\
\nabla _{X_{i}}\xi _{\alpha } &=&X_{i},%
\begin{array}{c}
\end{array}%
\nabla _{Y_{i}}\xi _{\alpha }=Y_{i},1\leq i,j\leq n,1\leq {\alpha }\leq s.
\end{eqnarray*}
\end{example}

We construct an example of generalized Kenmotsu manifold for $7-$dimensional.

\begin{example}
Let $n=2$ and $s=3$. The vector fields%
\begin{equation*}
e_{1}=f_{1}(z_{1},z_{2},z_{3})\frac{\partial }{\partial x_{1}}%
+f_{2}(z_{1},z_{2},z_{3})\frac{\partial }{\partial y_{1}},
\end{equation*}%
\begin{equation*}
e_{2}=-f_{2}(z_{1},z_{2},z_{3})\frac{\partial }{\partial x_{1}}%
+f_{1}(z_{1},z_{2},z_{3})\frac{\partial }{\partial y_{1}},
\end{equation*}%
\begin{equation*}
e_{3}=f_{1}(z_{1},z_{2},z_{3})\frac{\partial }{\partial x_{2}}%
+f_{2}(z_{1},z_{2},z_{3})\frac{\partial }{\partial y_{2}},
\end{equation*}%
\begin{equation*}
e_{4}=-f_{2}(z_{1},z_{2},z_{3})\frac{\partial }{\partial x_{2}}%
+f_{1}(z_{1},z_{2},z_{3})\frac{\partial }{\partial y_{2}},
\end{equation*}%
\begin{equation*}
e_{5}=\frac{\partial }{\partial z_{1}},e_{6}=\frac{\partial }{\partial z_{2}}%
,e_{7}=\frac{\partial }{\partial z_{3}}
\end{equation*}%
where $f_{1}$ and $f_{2}$ are given by
\begin{eqnarray*}
f_{1}(z_{1},z_{2},z_{3}) &=&c_{2}e^{-(z_{1}+z_{2}+z_{3})}\cos
(z_{1}+z_{2}+z_{3})-c_{1}e^{-(z_{1}+z_{2}+z_{3})}\sin (z_{1}+z_{2}+z_{3}), \\
f_{2}(z_{1},z_{2},z_{3}) &=&c_{1}e^{-(z_{1}+z_{2}+z_{3})}\cos
(z_{1}+z_{2}+z_{3})+c_{2}e^{-(z_{1}+z_{2}+z_{3})}\sin (z_{1}+z_{2}+z_{3})
\end{eqnarray*}%
for nonzero constant $c_{1},c_{2}.$ It is obvious that $\left\{
e_{1},e_{2},e_{3},e_{4},e_{5},e_{6},e_{7}\right\} $ are linearly independent
at each point of $M$. Let $g$ be the Riemannian metric given by
\begin{equation*}
g=\frac{1}{f_{1}^{2}+f_{2}^{2}}\sum_{i=1}^{2}(dx_{i}\otimes
dx_{i}+dy_{i}\otimes dy_{i})+dz_{1}\otimes dz_{1}+dz_{2}\otimes
dz_{2}+dz_{3}\otimes dz_{3},
\end{equation*}%
where $\left\{ x_{1},y_{1},x_{2},y_{2},z_{1},z_{2},z_{3}\right\} $ are
standard coordinates in $%
\mathbb{R}
^{7}$. Let $\eta ^{1}$, $\eta ^{2}$ and $\eta ^{3}$ be the $1-$form defined
by $\eta ^{1}(X)=g(X,e_{5})$, $\eta ^{2}(X)=g(X,e_{6})$ and $\eta
^{3}(X)=g(X,e_{7})$, respectively, for any vector field $X$ on $M$ and $\phi
$ be the $(1,1)$ tensor field defined by
\begin{eqnarray*}
\varphi (e_{1}) &=&e_{2},%
\begin{array}{cc}
\begin{array}{c}
\end{array}
&
\end{array}%
\varphi (e_{2})=-e_{1},%
\begin{array}{cc}
&
\end{array}%
\varphi (e_{3})=e_{4},%
\begin{array}{cc}
\begin{array}{c}
\end{array}
&
\end{array}%
\varphi (e_{4})=-e_{3}, \\
\varphi (e_{5} &=&\xi _{1})=0,%
\begin{array}{cc}
\begin{array}{c}
\end{array}
&
\end{array}%
\varphi (e_{6}=\xi _{2})=0,%
\begin{array}{cc}
\begin{array}{c}
\end{array}
&
\end{array}%
\varphi (e_{7}=\xi _{3})=0.
\end{eqnarray*}%
Therefore, the essential non-zero component of $\Phi $ is%
\begin{equation*}
\Phi (\frac{\partial }{\partial x_{i}},\frac{\partial }{\partial y_{i}})=-%
\frac{1}{f_{1}^{2}+f_{2}^{2}}=-\frac{2e^{2(z_{1}+z_{2}+z_{3})}}{%
c_{1}^{2}+c_{2}^{2}},%
\begin{array}{cc}
&
\end{array}%
i=1,2
\end{equation*}%
and hence
\begin{equation*}
\Phi =-\frac{2e^{2(z_{1}+z_{2}+z_{3})}}{c_{1}^{2}+c_{2}^{2}}%
\sum_{i=1}^{2}dx_{i}\wedge dy_{i}.
\end{equation*}%
Thus, we have $\eta ^{1}\wedge ...\wedge \eta ^{s}\wedge \Phi ^{n}\neq 0$ on
$M$. Consequently, the exterior derivative $d\Phi $ is given by
\begin{equation*}
d\Phi =-\frac{4e^{2(z_{1}+z_{2}+z_{3})}}{c_{1}^{2}+c_{2}^{2}}%
(dz_{1}+dz_{2}+dz_{3})\wedge \sum_{i=1}^{2}dx_{i}\wedge dy_{i}.
\end{equation*}%
Since $\eta ^{1}=dz_{1}$, $\eta ^{2}=dz_{2}$ and $\eta ^{3}=dz_{3},$ we find%
\begin{equation*}
d\Phi =2(\eta ^{1}+\eta ^{2}+\eta ^{3})\wedge \Phi .
\end{equation*}%
In addition, Nijenhuis tersion of $\varphi $ is equal to zero.
\end{example}

\begin{theorem}
Let $\left( M,\varphi ,\xi _{i},\eta ^{i},g\right)$ be an almost $s$-contact metric manifold. $M$ is a generalized Kenmotsu manifold if and only if%
\begin{equation}
\left( \nabla _{X}\varphi \right) Y=\underset{i=1}{\overset{s}{\sum }}%
\left\{ g(\varphi X,Y)\xi _{i}-\eta ^{i}(Y)\varphi X\right\}
\end{equation}%
for all $X,Y\in \Gamma (TM),$ $i\in \left\{ 1,2,...,s\right\} ,$ where $%
\nabla $ is Riemannian connection on M.
\end{theorem}
\begin{proof}
Let $M$ be a generalized Kenmotsu manifold. From $(1)$, $(6),(7)$ and $%
(8)$ \ for all $X,Y\in \Gamma (TM),$ we have%
\begin{eqnarray*}
g\left( \left( \nabla _{X}\varphi \right) Y,Z\right) &=&3\left\{ \underset{%
i=1}{\overset{s}{\sum }}(\eta ^{i}\wedge \Phi )(X,\varphi Y,\varphi Z)-%
\underset{i=1}{\overset{s}{\sum }}(\eta ^{i}\wedge \Phi )(X,Y,Z)\right\} \\
&=&\underset{i=1}{\overset{s}{3\sum }}\{\frac{1}{3}(-\eta ^{i}(X)\Phi
(\varphi Y,\varphi Z)+\eta ^{i}(\varphi Y)\Phi (\varphi Z,X)+\eta
^{i}(\varphi Z)\Phi (X,\varphi Y)) \\
&&-\frac{1}{3}(-\eta ^{i}(X)\Phi (Y,Z)+\eta ^{i}(Y)\Phi (Z,X)+\eta
^{i}(Z)\Phi (X,Y))\} \\
&=&\underset{i=1}{\overset{s}{\sum }}\left\{ -\eta ^{i}(X)g(\varphi
Y,\varphi ^{2}Z)+\eta ^{i}(X)g(Y,\varphi Z)-\eta ^{i}(Y)g(Z,\varphi X)-\eta
^{i}(Z)g(X,\varphi Y)\right\} \\
&=&\underset{i=1}{\overset{s}{\sum }}\left\{ -\eta ^{i}(Y)g(Z,\varphi
X)-\eta ^{i}(Z)g(X,\varphi Y)\right\} \\
&=&g(\underset{i=1}{\overset{s}{\sum }}\left\{ g(\varphi X,Y)\xi _{i}-\eta
^{i}(Y)\varphi X\right\} ,Z).
\end{eqnarray*}

Conversely, firstly, using $(9)$ and $(8)$, we get%
\begin{equation*}
\varphi \nabla _{X}\xi _{j}=\underset{i=1}{\overset{s}{-\sum }}\left\{
g(\varphi X,\xi _{j})\xi _{i}-\eta ^{i}(\xi _{j})\varphi X\right\}
\end{equation*}%
hence, we\ get%
\begin{equation*}
\varphi ^{2}\nabla _{X}\xi _{j}=\varphi ^{2}X.
\end{equation*}%
Therefore, we have%
\begin{equation*}
\nabla _{X}\xi _{j}=-\varphi ^{2}X.
\end{equation*}%
On the other hand, we get
\begin{equation*}
d\eta ^{i}(X,Y)=\frac{1}{2}\{g(Y,-\varphi ^{2}X)-g(X,-\varphi ^{2}Y)\}=0
\end{equation*}%
for all $X,Y\in \Gamma (TM),$ $i\in \left\{ 1,2,...,s\right\} .$ In
addition, we know%
\begin{eqnarray*}
3d\Phi (X,Y,Z) &=&Xg(Y,\varphi Z)-Yg(X,\varphi Z)-Zg(X,\varphi
Y)-g([X,Y],\varphi Z)\\&&+g([X,Z],\varphi Y)-g([Y,Z],\varphi X) \\
&=&g(Y,\nabla _{X}\varphi Z-\varphi \nabla _{X}Z)-g(X,\nabla _{Y}\varphi
Z-\varphi \nabla _{Y}Z)+g(X,\nabla _{Z}\varphi Y-\varphi \nabla _{Z}Y).
\end{eqnarray*}%
From hypothesis, we have%
\begin{eqnarray*}
3d\Phi (X,Y,Z) &=&\overset{s}{\underset{i=1}{\sum }}\{g(\varphi X,Z)g(Y,\xi
_{i})-\eta ^{i}(Z)g(Y,\varphi X)-g(\varphi Y,Z)g(X,\xi _{i})+\eta
^{i}(Z)g(X,\varphi Y) \\
&&+g(\varphi Z,Y)g(X,\xi _{i})-\eta ^{i}(Y)g(X,\varphi Z)\} \\
&=&2\overset{s}{\underset{i=1}{\sum }}\{\Phi (Z,X)\eta ^{i}(Y)+\Phi
(X,Y)\eta ^{i}(Z)+\Phi (Y,Z)\eta ^{i}(X)\}.
\end{eqnarray*}%
Then, we have,%
\begin{equation*}
d\Phi =2\underset{i=1}{\overset{s}{\sum }}\eta ^{i}\wedge \Phi .
\end{equation*}%
Moreover, the Nijenhuis torsion of $\varphi $ is obtained%
\begin{eqnarray*}
N_{\varphi }(X,Y) &=&\varphi \left( -\overset{s}{\underset{i=1}{\sum }}%
\{g(\varphi X,Y)\xi _{i}-\eta ^{i}(Y)\varphi X\}+\overset{s}{\underset{i=1}{%
\sum }}\{g(\varphi Y,X)\xi _{i}-\eta ^{i}(X)\varphi Y\}\right) \\
&&+\overset{s}{\underset{i=1}{\sum }}\{g(\varphi ^{2}X,Y)\xi _{i}-\eta
^{i}(Y)\varphi ^{2}X\}-\overset{s}{\underset{i=1}{\sum }}\{g(\varphi
^{2}Y,X)\xi _{i}-\eta ^{i}(X)\varphi ^{2}Y\} \\
&=&0.
\end{eqnarray*}%
Hence, we have%
\begin{equation*}
\lbrack \varphi ,\varphi ]+2\sum_{i=1}^{s}d\eta ^{i}\otimes \xi _{i}=0.
\end{equation*}%
The proof is completed.
\end{proof}

\begin{corollary}
Let $M$ be a $(2n+s)$-dimensional generalized Kenmotsu manifold with
structure $\left( \varphi ,\xi _{i},\eta ^{i},g\right) .$ Then we have%
\begin{equation}
\nabla _{X}\xi _{j}=-\varphi ^{2}X
\end{equation}%
for all $X\in \Gamma (TM),i,$ $j\in \left\{ 1,2,...,s\right\} .$
\end{corollary}

\begin{lemma}
Let $M$ be a $(2n+s)$-dimensional generalized Kenmotsu manifold with
structure $\left( \varphi ,\xi _{i},\eta ^{i},g\right) .$ Then we have%
\begin{equation*}
i)\nabla _{\xi _{j}}\varphi =0,%
\begin{array}{cc}
\begin{array}{c}
\end{array}
&
\end{array}%
\nabla _{\xi _{j}}\xi _{i}=0
\end{equation*}%
\begin{equation*}
ii)(L_{\xi _{i}}\varphi )X=0,%
\begin{array}{cc}
\begin{array}{c}
\end{array}
&
\end{array}%
(L_{\xi _{i}}\eta ^{j})X=0
\end{equation*}%
\begin{equation}
iii)(L_{\xi _{i}}g)(X,Y)=2\{g(X,Y)-\sum_{i=1}^{s}\eta ^{i}(X)\eta ^{i}(Y)\}
\end{equation}%
for all $X\in \Gamma (TM),$ $i,j\in \left\{ 1,2,...,s\right\} .$
\end{lemma}

\begin{theorem}
Let $M$ be a $(2n+s)$-dimensional generalized Kenmotsu manifold with
structure $\left( \varphi ,\xi _{i},\eta ^{i},g\right) .$ Then we have%
\begin{equation}
(\nabla _{X}\eta ^{i})Y=g(X,Y)-\underset{j=1}{\overset{s}{\sum }}\eta
^{j}(X)\eta ^{j}(Y)
\end{equation}
\end{theorem}

for all $X,Y\in \Gamma (TM),i\in \left\{ 1,2,...,s\right\} .$

\begin{proof}
Using $(8)$ and $(10)$ we get the desired result.
\end{proof}

We have below the corollary in case $s=1$.

\begin{corollary}
Let $\left( M^{2n+1},\varphi ,\xi ,\eta
,g\right) $ be an almost contact metric manifold. $M$ is a Kenmotsu manifold if and only if
\begin{equation*}
\left( \nabla _{X}\varphi \right) Y=g(\varphi X,Y)\xi -\eta (Y)\varphi X
\end{equation*}%
for all $X,Y\in \Gamma (TM),$ $i\in \left\{ 1,2,...,s\right\} ,$ where $%
\nabla $ is Riemannian connection on M \textit{\cite{K}}.
\end{corollary}

\begin{theorem}
Let $F$ be \ a K\"{a}hler manifold $f(t)=ke^{\overset{s}{\underset{i=1}{\sum
}}t_{i}}$ be\ a function on $
\mathbb{R}^{s},$ and k be a non-zero constant. Then the warped product space 
$M=\mathbb{R}^{s}\times _{f}F$ have a generalized Kenmotsu manifold.
\end{theorem}

\begin{proof}
\bigskip Let $(F,J,G)$ be a K\"{a}hler manifold and consider 
$M=\mathbb{R}^{s}\times _{f}F,$ with coordinates $(t_{1},...,t_{s},x_{1},...,x_{2n})$. We
define $\varphi $ tensor field, $1$-form $\eta ^{i},$ vector field $\xi _{i}$
and Riemannian metric tensor $g$ on $M$ as follows:
\begin{eqnarray*}
\varphi (\frac{\partial }{\partial t_{i}},U) &=&(0,JU), \\
\eta ^{j}(\frac{\partial }{\partial t_{i}},U) &=&\delta _{ij},
\begin{array}{cc}
&
\end{array}
\xi _{i}=(\frac{\partial }{\partial t_{i}},0) \\
g_{f} &=&\overset{s}{\underset{i=1}{\sum }}dt^{i}\otimes dt^{i}+f^{2}\pi
^{\ast }(G)
\end{eqnarray*}%
where $f(t)=ke^{\overset{s}{\underset{i=1}{\sum }}t_{i}},U\in \Gamma (F)$.

Then $(M,\varphi ,\eta ^{i},\xi _{i},g_{f})$ defines $s-$contact metric
manifold. Now let us show that this manifold is a generalized Kenmotsu
manifold.

It is clear that $\eta ^{i}$ are closed. Thus, we have%
\begin{equation*}
\Phi (X,Y)=g_{f}(X,\varphi Y)=f^{2}\pi ^{\ast }(G(X,JY))
\end{equation*}

or
\begin{equation*}
\Phi =f^{2}\pi ^{\ast }(\Psi )
\end{equation*}%
where \ is fundamental $2$-form of K$\ddot{a}$hler manifold. Hence, we get

\begin{equation*}
d\Phi =2c\overset{s}{\underset{i=1}{\sum }}e^{2\overset{s}{\underset{i=1}{%
\sum }}t_{i}}dt^{i}\wedge \pi ^{\ast }(\Psi )=2\overset{s}{\underset{i=1}{%
\sum }}dt^{i}\wedge \Phi .
\end{equation*}

Finally torsion tensor $N_{\varphi }$ of $M$ is vanish, since $\eta ^{i}$
are closed and $N_{J}=0$.

Then $(M=%
\mathbb{R}
^{s}\times _{f}F,\varphi ,\eta ^{i},\xi _{i},g_{f})$ is a generalized
Kenmotsu manifold.
\end{proof}

\begin{example}
$(\mathbb{R}^{2}\times _{f}V^{4},g_{f}=\overset{2}{\underset{i=1}{\sum }}dt^{i}\otimes
dt^{i}+f^{2}G)$ is warped product with coordinates $%
(t_{1},t_{2},x_{1},x_{2},x_{3},x_{4})$, where $f^{2}=k^{2}e^{\overset{2}{%
\underset{i=1}{\sum }}t_{i}}$. Take a orthonormal frame field $\{\overline{E}%
_{1},\overline{E}_{2},\overline{E}_{3},\overline{E}_{4}\}$ of $V^{4}$ and $\{%
\overline{e}_{5},\overline{e}_{6}\}$ of $
\mathbb{R}^{2}$ such that $\overline{E}_{2}=J\overline{E}_{1}$, $\overline{E}_{4}=J%
\overline{E}_{3}.$Then we obtain a local orthonormal field $%
\{E_{1},E_{2},E_{3},E_{4},E_{5},E_{6}\}$ of $
\mathbb{R}^{2}\times _{f}V^{4}$ by%
\begin{eqnarray*}
E_{1} &=&ke^{-\overset{2}{\underset{i=1}{\sum }}t_{i}}\overline{E}_{1},%
\begin{array}{cc}
&
\end{array}%
E_{2}=ke^{-\overset{2}{\underset{i=1}{\sum }}t_{i}}\overline{E}_{2} \\
E_{3} &=&-ke^{-\overset{2}{\underset{i=1}{\sum }}t_{i}}\overline{E}_{3},%
\begin{array}{cc}
&
\end{array}%
E_{4}=-ke^{-\overset{2}{\underset{i=1}{\sum }}t_{i}}\overline{E}_{4} \\
E_{5} &=&\xi _{1},%
\begin{array}{cc}
&
\end{array}%
E_{6}=\xi _{2}.
\end{eqnarray*}%
Then $
\mathbb{R}^{2}\times _{f}V^{4}$ is a generalized Kenmotsu manifold.
\end{example}

\begin{theorem}
Let $(M^{2n+s},\varphi ,\eta ^{i},\xi _{i},g)$ be a generalized Kenmotsu
manifold, $V$ and $L$ are K\"{a}hler and a flat manifold with locally
coordinates $(x_{1},...,x_{2n})$ and $(t_{1},...,t_{s})$\ \ respectively.
Then $M$ \ a locally warped product $L^{s}\times _{f}V^{2n}$ \ where $%
f(t)=ke^{\overset{s}{\underset{i=1}{\sum }}t_{i}}$ and $k$ a nonzero
positive constant.
\end{theorem}

\begin{proof}
We know that $TM=\mathcal{L}\oplus \mathcal{M%
}$. $ \mathcal{L}\ $ is clearly integrable, since $d\eta ^{i}=0.$ Then $V$ integral manifold of $ \mathcal{L}\ $ is totally umbilical because $\nabla _{X}\xi _{i}=X.$ On the other hand $[\xi _{i},\xi _{j}]=0$ and $\nabla _{\xi _{i}}\xi _{j}=0$, $\mathcal{M}$ is integrable and $L$ integral manifold is\ totally geodesic.

We select $J=\varphi \mid _{D}$ such that $J^{2}=-I,$ $G=g\mid _{D}.$ Then $%
(V,J,G)$ is almost Hermitian manifold. Also torsion tensor $N_{J}=N_{\varphi
}=0$ and using $(9),$ we get $(\nabla _{X}J)Y=0.$ Then \ $(V,J,G)$ is K%
\"{a}hler manifold.

Then $M=L\times _{f}V$ is locally a warped product \ and metric \ is
\begin{equation*}
g_{f}=\overset{s}{\underset{i=1}{\sum }}dt^{i}\otimes dt^{i}+f^{2}G.
\end{equation*}%
Its follows that
\begin{equation*}
(L_{\xi _{i}}g_{f})(X,Y)=\frac{2\xi _{i}(f)}{f}G(X,Y)
\end{equation*}%
and using $(11)$, we get
\begin{equation*}
\xi _{i}(f)=f,%
\begin{array}{cc}
&
\end{array}%
i=1,...,s.
\end{equation*}%
Thus,we have
\begin{equation*}
\frac{\partial f(t_{1},...,t_{s})}{\partial t_{i}}=f(t_{1},...,t_{s}),%
\begin{array}{cc}
&
\end{array}%
i=1,...,s.
\end{equation*}%
Therefore, we obtained $f(t_{1},...,t_{s})=ce^{\overset{s}{\underset{i=1}{%
\sum }}t_{i}}$ where c is nonzero constant.
\end{proof}
\begin{example}
\bigskip Let's go back to the example 2.3. Let $(%
\mathbb{R}
^{7},\varphi ,\eta ^{i},\xi _{i},g)$ be a generalized Kenmotsu manifold
where $i=1,2,3$. Take a orthonormal frame field
\begin{equation*}
\left\{ \frac{\partial }{\partial z_{1}}=\xi _{1},%
\begin{array}{c}
\end{array}%
\frac{\partial }{\partial z_{2}}=\xi _{2},%
\begin{array}{c}
\end{array}%
\frac{\partial }{\partial z_{3}}=\xi _{3}\right\}
\end{equation*}%
of $%
\mathbb{R}
^{3}$ and
\begin{equation*}
\left\{ \frac{e^{2\overset{3}{\underset{i=1}{\sum }}z_{i}}}{%
c_{1}^{2}+c_{2}^{2}}(f_{1}e_{1}-f_{2}e_{2}),%
\begin{array}{c}
\end{array}%
\frac{e^{2\overset{3}{\underset{i=1}{\sum }}z_{i}}}{c_{1}^{2}+c_{2}^{2}}%
(f_{2}e_{1}+f_{1}e_{2}),%
\begin{array}{c}
\end{array}%
\frac{e^{2\overset{3}{\underset{i=1}{\sum }}z_{i}}}{c_{1}^{2}+c_{2}^{2}}%
(f_{1}e_{3}-f_{2}e_{4}),%
\begin{array}{c}
\end{array}%
\frac{e^{2\overset{3}{\underset{i=1}{\sum }}z_{i}}}{c_{1}^{2}+c_{2}^{2}}%
(f_{2}e_{3}+f_{1}e_{4})\right\}
\end{equation*}%
of $%
\mathbb{R}
^{4}.$ Then $%
\mathbb{R}
^{7}=%
\mathbb{R}
^{3}\times
\mathbb{R}
^{4}$ is product manifold, the structure by tensor $\varphi $ and metric
tensor $g.$ $%
\mathbb{R}
^{4}$ is the standard K\"{a}hler structure $(J,G)$. Here the Riemannian
metric $g$ is warped product metric
\begin{equation*}
g_{0}+cf^{2}G
\end{equation*}%
where $g_{0}$\ \ is the Euclidean metric of $%
\mathbb{R}
^{3},$ f is the function defined on $%
\mathbb{R}
^{3}$ by%
\begin{equation*}
f(z_{1},z_{2},z_{3})=e^{2\overset{3}{\underset{i=1}{\sum }}z_{i}}\text{and}%
\begin{array}{cc}
&
\end{array}%
c=\frac{1}{c_{1}^{2}+c_{2}^{2}}.
\end{equation*}
\end{example}
\section{Some Curvature Properties}
\begin{theorem}
Let $M$ be a $(2n+s)$-dimensional generalized Kenmotsu manifold with
structure $\left( \varphi ,\xi _{i},\eta ^{i},g\right) .$ Then we have%
\begin{equation}
R(X,Y)\xi _{i}=\underset{j=1}{\overset{s}{\sum }}\{\eta ^{j}(Y)\varphi
^{2}X-\eta ^{j}(X)\varphi ^{2}Y\}
\end{equation}%
for all $X,Y\in \Gamma (TM),i\in \left\{ 1,2,...,s\right\} .$
\end{theorem}
\begin{proof}
Firstly, using $(10)$ and $(6)$ we get%
\begin{equation*}
\nabla _{X}\nabla _{Y}\xi _{i}=\nabla _{X}Y+\varphi ^{2}X\underset{j=1}{%
\overset{s}{\sum }}\eta ^{j}(Y)-\underset{j=1}{\overset{s}{\sum }}\{\eta
^{j}(\nabla _{X}Y)\xi _{j}+g(Y,-\varphi ^{2}X)\xi _{j}\}
\end{equation*}%
and%
\begin{equation*}
\nabla _{\lbrack X,Y]}\xi _{i}=-\varphi ^{2}\nabla _{X}Y+\varphi ^{2}\nabla
_{Y}X.
\end{equation*}%
Then,%
\begin{eqnarray*}
R(X,Y)\xi _{i} &=&\nabla _{X}Y+\varphi ^{2}X\underset{j=1}{\overset{s}{\sum }%
}\eta ^{j}(Y)-\underset{j=1}{\overset{s}{\sum }}\{\eta ^{j}(\nabla _{X}Y)\xi
_{j}+g(Y,-\varphi ^{2}X)\xi _{j}\} \\
&&-\nabla _{Y}X-\varphi ^{2}Y\underset{j=1}{\overset{s}{\sum }}\eta ^{j}(X)+%
\underset{j=1}{\overset{s}{\sum }}\{\eta ^{j}(\nabla _{Y}X)\xi
_{j}+g(X,-\varphi ^{2}Y)\xi _{j}\} \\
&&+\varphi ^{2}\nabla _{X}Y-\varphi ^{2}\nabla _{Y}X.
\end{eqnarray*}%
From $(6)$ desired result.
\end{proof}
\begin{corollary}
Let $M$ be a $(2n+s)$-dimensional generalized Kenmotsu manifold with
structure $\left( \varphi ,\xi _{i},\eta ^{i},g\right) .$ Then we have%
\begin{equation}
R(X,\xi _{i})Y=\underset{j=1}{\overset{s}{\sum }}\{\eta ^{j}(Y)\varphi
^{2}X-g(X,\varphi ^{2}Y)\xi _{j}\}
\end{equation}%
\begin{equation}
R(X,\xi _{j})\xi _{i}=\varphi ^{2}X,%
\begin{array}{cc}
\begin{array}{c}
\end{array}
&
\end{array}%
R(\xi _{k},\xi _{j})\xi _{i}=0
\end{equation}%
for all $X,Y\in \Gamma (TM),i,j,k\in \left\{ 1,2,...,s\right\} .$
\end{corollary}
\begin{corollary}
\textit{\cite{K} Let }$M$ be a $(2n+1)$-dimensional Kenmotsu manifold with
structure $\left( \varphi ,\xi ,\eta ,g\right) .$ Then we have%
\begin{eqnarray*}
R(X,Y)\xi &=&\eta (Y)X-\eta (X)Y \\
R(X,\xi )Y &=&g(X,Y)\xi -\eta (Y)X,%
\begin{array}{cc}
&
\end{array}%
R(\xi ,\xi )\xi =0
\end{eqnarray*}%
for all $X,Y\in \Gamma (TM).$
\end{corollary}
\begin{theorem}
Let $M$ be a $(2n+s)$-dimensional generalized Kenmotsu manifold with
structure $\left( \varphi ,\xi _{i},\eta ^{i},g\right) .$ Then we have%
\begin{eqnarray*}
(\nabla _{Z}R)(X,Y,\xi _{i}) &=&sg(Z,X)Y-sg(Z,Y)X-R(X,Y)Z \\
&&+s\underset{h=1}{\overset{s}{\sum }}\eta ^{h}(Z)\{\eta ^{h}(Y)X-\eta
^{h}(X)Y\}+\underset{l=1}{\overset{s}{\sum }}\eta ^{l}(Z)R(X,Y)\xi _{l}
\end{eqnarray*}%
for all $X,Y\in \Gamma (TM),i\in \left\{ 1,2,...,s\right\} .$
\end{theorem}
\begin{proof}
Using $(10)$ and $(13)$, we have%
\begin{eqnarray*}
(\nabla _{Z}R)(X,Y,\xi _{i}) &=&\nabla _{Z}\{\underset{j=1}{\overset{s}{\sum
}}\{\eta ^{j}(X)Y-\eta ^{j}(Y)X\}\}-\underset{j=1}{\overset{s}{\sum }}\{\eta
^{j}(\nabla _{Z}X)Y-\eta ^{j}(Y)\nabla _{Z}X\} \\
&&\text{ \ \ }-\underset{j=1}{\overset{s}{\sum }}\{\eta ^{j}(X)\nabla
_{Z}Y-\eta ^{j}(\nabla _{Z}Y)X\}-R(X,Y)\varphi ^{2}Z.
\end{eqnarray*}%
From $(6)$, we get%
\begin{equation*}
(\nabla _{Z}R)(X,Y,\xi _{i})=\underset{j=1}{\overset{s}{\sum }}\{g(X,\nabla
_{Z}\xi _{j})Y-g(Y,\nabla _{Z}\xi _{j})X\}-R(X,Y)Z+\underset{k=1}{\overset{s}%
{\sum }}\eta ^{k}(Y)R(X,Y)\xi _{k}.
\end{equation*}%
The proof competes from $(6)$ and $(10)$.
\end{proof}
\begin{corollary}
Let $M$ be a $(2n+s)$-dimensional generalized Kenmotsu manifold Then we have%
\begin{eqnarray*}
(\nabla _{Z}R)(X,Y,\xi _{i}) &=&sg(Z,X)Y-sg(Z,Y)X-R(X,Y)Z,%
\begin{array}{cc}
&
\end{array}%
Z\in \mathcal{L} \\
(\nabla _{\xi _{j}}R)(X,Y,\xi _{i}) &=&0
\end{eqnarray*}%
for all $X,Y\in \Gamma (TM),i\in \left\{ 1,2,...,s\right\} .$
\end{corollary}
\begin{corollary}
\textit{\cite{K} Let }$M$ be a $(2n+1)$-dimensional Kenmotsu manifold with
structure $\left( \varphi ,\xi ,\eta ,g\right) .$ Then we have%
\begin{equation*}
(\nabla _{Z}R)(X,Y,\xi )=g(Z,X)Y-g(Z,Y)X-R(X,Y)Z,%
\begin{array}{cc}
&
\end{array}%
Z\in \mathcal{L}\text{ and }(\nabla _{\xi }R)(X,Y,\xi )=0.
\end{equation*}
\end{corollary}
\begin{corollary}
Let $\left( M,\varphi ,\xi _{i},\eta ^{i},g\right) $ be a $(2n+s)$%
-dimensional locally-symmetric generalized Kenmotsu manifold. Then we have%
\begin{equation*}
R(X,Y)Z=s\{g(Z,X)Y-g(Z,Y)X\}.
\end{equation*}
\end{corollary}
\begin{corollary}
\textit{\cite{K} Let }M be a $(2n+1)$-dimensional Kenmotsu manifold with
structure $\left( \varphi ,\xi ,\eta ,g\right) .$ If M is a locally
symmetric then we have%
\begin{equation*}
R(X,Y)Z=g(Z,X)Y-g(Z,Y)X.
\end{equation*}
\end{corollary}
\begin{corollary}
The $\varphi -$sectional curvature of any locally symmetric generalized
Kenmotsu manifold $\left( M,\varphi ,\xi _{i},\eta ^{i},g\right) $ is equal
to $-s$.
\end{corollary}
In this case $s=1$, we obtain that the $\varphi -$sectional curvature of any
locally symmetric Kenmotsu manifold $\left( M,\varphi ,\xi ,\eta ,g\right) $
is equal to $-1$ \textit{\cite{K}.}
\begin{theorem}
Let $M$ be a $(2n+s)$-dimensional generalized Kenmotsu manifold with
structure $\left( \varphi ,\xi _{i},\eta ^{i},g\right) .$ Then we have%
\begin{equation*}
R(X,Y)\varphi Z-\varphi R(X,Y)Z=g(Y,Z)\varphi X-g(X,Z)\varphi Y-g(Y,\varphi
Z)X+g(X,\varphi Z)Y
\end{equation*}%
\begin{equation*}
R(\varphi X,\varphi Y)Z=R(X,Y)Z+g(Y,Z)X-g(X,Z)Y+g(Y,\varphi Z)\varphi
X-g(X,\varphi Z)\varphi Y
\end{equation*}%
for all $X,Y\in \Gamma (TM),i\in \left\{ 1,2,...,s\right\} .$
\end{theorem}
\section{Ricci Curvature Tensor}
\begin{theorem}
Let $M$ be a $(2n+s)$-dimensional generalized Kenmotsu manifold with
structure $\left( \varphi ,\xi _{i},\eta ^{i},g\right) .$ Then we have%
\begin{equation}
S(X,\xi _{i})=-2n\underset{j=1}{\overset{s}{\sum }}\eta ^{j}(X)
\end{equation}%
for all $X,Y\in \Gamma (TM),i\in \left\{ 1,2,...,s\right\} .$
\end{theorem}
\begin{proof}
If $\{E_{1},E_{2},...,E_{2n+s}\}$ are local orthonormal vector fields, then $%
S(X,Y)=\underset{k=1}{\overset{2n+s}{\sum }}g(R(E_{k},X)Y,E_{k})$ defines a
global tensor field $S$ of type $(0,2)$. Then, we obtain%
\begin{eqnarray*}
S(X,\xi _{i}) &=&\underset{k=1}{\overset{2n}{\sum }}g(\underset{j=1}{\overset%
{s}{\sum }}\{\eta ^{j}(X)\varphi ^{2}E_{k}-\eta ^{j}(E_{k})\varphi
^{2}X\},E_{k})+\underset{k=1}{\overset{s}{\sum }}g(-\varphi ^{2}X,\xi _{k})
\\
&=&\underset{j=1}{\overset{s}{\sum }}\eta ^{j}(X)\underset{k=1}{\overset{2n}{%
\sum }}g(\varphi ^{2}E_{k},E_{k}).
\end{eqnarray*}
\end{proof}
In this case $s=1$ we have $S(X,\xi )=-2n\eta (X)$ in \cite{K}.
\begin{corollary}
Let $M$ be a $(2n+s)$-dimensional generalized Kenmotsu manifold with
structure $\left( \varphi ,\xi _{i},\eta ^{i},g\right) .$ Then we have%
\begin{equation}
S(\xi _{k},\xi _{i})=-2n
\end{equation}%
for all $X,Y\in TM,i,k\in \left\{ 1,2,...,s\right\} .$
\end{corollary}
\begin{theorem}
Let $M$ be a $(2n+s)$ dimensional generalized Kenmotsu manifold with
structure $\left( \varphi ,\xi _{i},\eta ^{i},g\right) .$ Then we have%
\begin{equation}
S(\varphi X,\varphi Y)=S(X,Y)+2n\underset{i=1}{\overset{s}{\sum }}\eta
^{i}(X)\eta ^{i}(Y)
\end{equation}%
for all $X,Y\in \Gamma (TM),i\in \left\{ 1,2,...,s\right\} .$
\end{theorem}
\begin{proof}
We can put
\begin{equation*}
X=X_{0}+\underset{i=1}{\overset{s}{\sum }}\eta ^{i}(X)\xi _{i}\text{ and }%
Y=Y_{0}+\underset{i=1}{\overset{s}{\sum }}\eta ^{i}(Y)\xi _{i}
\end{equation*}
where $X_{0},Y_{0}\in \mathcal{L}$. Then from (16) and (17) we have,%
\begin{eqnarray*}
S(X,Y) &=&S(X_{0},Y_{0})+\underset{i=1}{\overset{s}{\sum }}\eta ^{i}(Y)\eta
^{i}(X_{0})+\underset{i=1}{\overset{s}{\sum }}\eta ^{i}(X)\eta ^{i}(Y_{0})+%
\underset{i=1}{\overset{s}{\sum }}\eta ^{i}(X)\eta ^{i}(Y)S(\xi _{i},\xi
_{i}) \\
&=&S(X_{0},Y_{0})-2n\underset{i=1}{\overset{s}{\sum }}\eta ^{i}(X)\eta
^{i}(Y).
\end{eqnarray*}
Since $\varphi X,\varphi Y\in \mathcal{L}$ we get $S(X_{0},Y_{0})=S(\varphi
X,\varphi Y)$ which implies the desired result.
\end{proof}
Considering $s=1$ in \textit{\cite{JDP},} we deduce
\begin{equation*}
S(\varphi X,\varphi Y)=S(X,Y)+2n\eta (X)\eta (Y).
\end{equation*}
\begin{theorem}
Let $M$ be a $(2n+s)$ dimensional generalized Kenmotsu manifold with
structure $\left( \varphi ,\xi _{i},\eta ^{i},g\right) .$ Then we have%
\begin{eqnarray*}
(\nabla _{\varphi X}S)(\varphi Y,\varphi Z) &=&(\nabla _{\varphi X}S)(Y,Z)-%
\underset{i=1}{\overset{s}{\sum }}\eta ^{i}(Y)\{S(X,\varphi Z)+2ng(X,\varphi
Z)\} \\
&&-\underset{i=1}{\overset{s}{\sum }}\eta ^{i}(Z)\{S(X,\varphi
Y)+2ng(X,\varphi Y)\}
\end{eqnarray*}%
for all $X,Y\in \Gamma (TM),i\in \left\{ 1,2,...,s\right\} .$
\end{theorem}
\begin{proof}
Using $(9)$, we get%
\begin{eqnarray*}
(\nabla _{\varphi X}S)(\varphi Y,\varphi Z) &=&\nabla _{\varphi X}S(Y,Z)+2n%
\underset{i=1}{\overset{s}{\sum }}\{\eta ^{i}(Y)\nabla _{\varphi X}\eta
^{i}(Z)+\eta ^{i}(Z)\nabla _{\varphi X}\eta ^{i}(Y)\} \\
&&+\underset{i=1}{\overset{s}{\sum }}\{-S(g(\varphi ^{2}X,Y)\xi _{i}-\eta
^{i}(Y)\varphi ^{2}X,\varphi Z)-S(\nabla _{\varphi X}Y,Z)-2n\eta ^{i}(\nabla
_{\varphi X}Y)\eta ^{i}(Z) \\
&&-S(\varphi Y,g(\varphi ^{2}X,Z)\xi _{i}-\eta ^{i}(Z)\varphi
^{2}X)-S(Y,\nabla _{\varphi X}Z)-2n\eta ^{i}(Y)\eta ^{i}(\nabla _{\varphi
X}Z)\}.
\end{eqnarray*}%
From $(6)$, $(16)$ and $(17)$ we have%
\begin{eqnarray*}
(\nabla _{\varphi X}S)(\varphi Y,\varphi Z) &=&(\nabla _{\varphi X}S)(Y,Z)+%
\underset{i=1}{\overset{s}{\sum }}\{2n\eta ^{i}(Y)(\nabla _{\varphi X}\eta
^{i})Z+2n\eta ^{i}(Z)(\nabla _{\varphi X}\eta ^{i})Y \\
&&-\eta ^{i}(Y)S(X,\varphi Z)-\eta ^{i}(Z)S(\varphi Y,X)\}.
\end{eqnarray*}
\end{proof}
\begin{corollary}
Let $M$ be a $(2n+s)$ dimensional generalized Kenmotsu manifold with
structure $\left( \varphi ,\xi _{i},\eta ^{i},g\right) .$ Then we have%
\begin{eqnarray*}
(\nabla _{X}S)(\varphi Y,\varphi Z) &=&(\nabla _{X}S)(Y,Z)+2n\underset{i=1}{%
\overset{s}{\sum }}\{g(X,Y)\eta ^{i}(Z)+g(X,Z)\eta ^{i}(Y)\} \\
&&+\underset{i=1}{\overset{s}{\sum }}\{\eta ^{i}(Y)S(X,Z)+\eta
^{i}(Z)S(X,Y)\}
\end{eqnarray*}%
for all $X,Y\in \Gamma (TM),i\in \left\{ 1,2,...,s\right\} .$
\end{corollary}
\begin{definition}
The Ricci tensor $S$ of a $(2n+s)$-dimensional generalized Kenmotsu manifold
$M$ is called $\eta -parallel$, if it satisfies%
\begin{equation*}
(\nabla _{X}S)(\varphi Y,\varphi Z)=0
\end{equation*}%
for all vector fields $X,Y$ and $Z$ on $M$.
\end{definition}
\begin{theorem}
Let $\left( M,\varphi ,\xi _{i},\eta ^{i},g\right) $ a $(2n+s)$-dimensional
generalized Kenmotsu manifold. $M$ has $\eta -parallel$ if and only if%
\begin{eqnarray*}
(\nabla _{X}S)(Y,Z) &=&-2n\underset{i=1}{\overset{s}{\sum }}\{g(X,Y)\eta
^{i}(Z)+g(X,Z)\eta ^{i}(Y)\} \\
&&-\underset{i=1}{\overset{s}{\sum }}\{\eta ^{i}(Y)S(X,Z)+\eta
^{i}(Z)S(X,Y)\}
\end{eqnarray*}%
for all $X,Y,Z\in \Gamma (TM),i\in \left\{ 1,2,...,s\right\} .$
\end{theorem}
\begin{corollary}
\cite{K} Let $\left( M,\varphi ,\xi ,\eta ,g\right) $ a $(2n+1)$-dimensional
Kenmotsu manifold. $M$ has $\eta -parallel$ if and only if%
\begin{eqnarray*}
(\nabla _{X}S)(Y,Z) &=&-2n\{g(X,Y)\eta (Z)+g(X,Z)\eta (Y)\} \\
&&\text{ \ \ \ }-\eta (Y)S(X,Z)-\eta (Z)S(X,Y)
\end{eqnarray*}%
for all $X,Y,Z\in \Gamma (TM).$
\end{corollary}
\section{Semi-Symmetric Properties of Generalized Kenmotsu Manifolds}
With respect to the Riemannian connection $\nabla $ of a generalized
Kenmotsu manifold $\left( M,\varphi ,\xi _{i},\eta ^{i},g\right) $, we can
prove:
\begin{theorem}
The $\varphi $- sectional curvature of any semi-symmetric $(2n+s)$%
-dimensional generalized Kenmotsu manifold $\left( M,\varphi ,\xi _{i},\eta
^{i},g\right) $ is equal to $-s$.
\end{theorem}
\begin{proof}
Let $X$ be a unit vector field. Since $\left( M,\varphi ,\xi _{i},\eta
^{i},g\right) $ is semi-symmetric, then
\begin{equation*}
(R.R)(X,\xi _{i},X,\varphi X,\varphi X,\xi _{i})=0,
\end{equation*}%
for any $i,j\in \left\{ 1,2,...,s\right\} .$ Expanding this formula from $%
(7)$ and taking into account $(14)$, we get
\begin{equation*}
R(X,\varphi X,\varphi X,X)=-s,
\end{equation*}%
which completes the proof.
\end{proof}
Observe that, in the case $s=1$, by using the \textit{Theorem 11} we obtain
that a semi-symmetric Kenmotsu manifolds is constant curvature equal to $-1$
\cite{Bin}.
\begin{theorem}
Let $\left( M,\varphi ,\xi _{i},\eta ^{i},g\right) $ be a $(2n+s)$
dimensional Ricci semi-symmetric generalized Kenmotsu manifold. Then its
Ricci tensor field $S$ respect the Riemannian connection satisfies%
\begin{equation}
S(X,Y)=-2n\{sg(\varphi X,\varphi Y)+\underset{i,j=1}{\overset{s}{\sum }}\eta
^{i}(X)\eta ^{j}(Y)\}
\end{equation}%
for any $X,Y\in \Gamma (TM).$
\end{theorem}
\begin{proof}
Since $\left( M,\varphi ,\xi _{i},\eta ^{i},g\right) $ is Ricci
semi-symmetric, then
\begin{equation*}
S(R(X,\xi _{i})\xi _{j},Y)+S(\xi _{j},R(X,\xi _{i})Y)=0,
\end{equation*}%
for any $X,Y\in \Gamma (TM)$ and $i,j\in \left\{ 1,2,...,s\right\} .$ Now,
from $(14),$ $(15)$ and $(16)$ we get the desired result.
\end{proof}
In this case $s=1$ we have following the corollary.
\begin{corollary}
Any Ricci semi-symmetric $(2n+1)-$dimensional Kenmotsu manifold is an
Einstein manifold.
\end{corollary}
\begin{proof}
Considering $s=1$ in $(19)$, we deduce
\begin{equation*}
S(X,Y)=-2ng(X,Y)
\end{equation*}%
for any $X,Y\in \Gamma (TM).$
\end{proof}
For the Weyl projective curvature tensor field $P$, the weyl projective
curvature tensor $P$ of a $(2n+s)$-dimensional generalized Kenmotsu manifold
$M$ is given by%
\begin{equation*}
P(X,Y)Z=R(X,Y)Z-\frac{1}{2n+s-1}\{S(Y,Z)X-S(X,Z)Y\}
\end{equation*}%
where $R$ is curvature tensor and $S$ is the ricci curvature tensor of $M$,
we have the following theorem.
\begin{theorem}
The $\varphi $- sectional curvature of any projectively semi-symmetric
generalized Kenmotsu manifold $\left( M,\varphi ,\xi _{i},\eta ^{i},g\right)
$ is equal to $-s$.
\end{theorem}
\begin{proof}
Let $X$ be a unit vector field. Then, from $(6)$ and taking into account $%
(14)$ and $(16)$ we have
\begin{equation*}
(R.P)(X,\xi _{i},X,\varphi X,\varphi X,\xi _{j})=(R.R)(X,\xi _{i},X,\varphi
X,\varphi X,\xi _{j}).
\end{equation*}%
This completes the proof from the \textit{Theorem 5.1}.
\end{proof}
\begin{corollary}
Let $\left( M,\varphi ,\xi ,\eta ,g\right) $ a $(2n+1)$-dimensional Kenmotsu
manifold. The $\varphi $- sectional curvature of any projectively
semi-symmetric Kenmotsu manifold if and only if $M$ is an Einstein manifold.
\end{corollary}

\bigskip 
\begingroup
\noindent%
\parbox[t]{7.8cm}{\small{\scshape\ignorespaces  Department of Mathematics\\ Faculty of Sciences\\ Gazi University\\ANKARA 06500, TURKEI 
}\par\vskip1ex
\noindent\small{\itshape E-mail address}\/: avanli@gazi.edu.tr\par\vskip4ex}%
\hfill\endgroup
\begingroup
\noindent%
\parbox[t]{7.8cm}{\small{\scshape\ignorespaces Department of Mathematics\\ Faculty of Sciences\\ Gazi University\\ANKARA 06500, TURKEI 
}\par\vskip1ex
\noindent\small{\itshape E-mail address}\/: ramazansr@gmail.com\par\vskip4ex}%
\hfill\endgroup

\end{document}